\documentclass{amsart}

\usepackage{cgeometry}

\bibliography{ComplexGeometry}
\DeclareMathOperator{\Supp}{Supp}

\begin{document}

\title[Integration by parts]{Integration by parts formula for non-pluripolar product}
\author{Mingchen Xia}
\date{\today}

\maketitle

\begin{abstract}
    In this paper, we prove the integration by parts formula for the non-pluripolar product on a compact K\"ahler manifold. Our result generalizes the special case of potentials with small unbounded loci proved in \cite{BEGZ10}.
\end{abstract}

\tableofcontents

\section{Introduction}
Let $X$ be a compact K\"ahler manifold of dimension $n$. Let $\alpha$ be a big cohomology class with a smooth representative $\theta$. Let $\PSH(X,\theta)$ be the space of $\theta$-psh functions on $X$. The mixed Monge--Amp\`ere operator is defined in \cite{BT82} for bounded potentials and in \cite{BEGZ10} for general potentials. The product constructed in \cite{BEGZ10} is called the \emph{non-pluripolar product}.

It is nature and important to know if one can perform integration by parts for the non-pluripolar product.
For potentials with small unbounded loci in $\PSH(X,\theta)$, one can always reduce the problem to the classical Bedford--Taylor theory with certain tricks and the integration by parts formula is proved in \cite{BEGZ10} Theorem~1.14 for these potentials.

The general case is not yet proved in literature. We fill the gap in this paper. 
\begin{theorem}
Let $\gamma_i,\varphi_j,\psi_j\in \PSH(X,\theta)$ ($j=1,2$, $i=1,\ldots,n-1$). Let $u=\varphi_1-\varphi_2$, $v=\psi_1-\psi_2$. Assume that 
\[
[\varphi_1]=[\varphi_2],\quad [\psi_1]=[\psi_2].
\]
Then
\begin{equation}\label{eq:1}
\int_X u\,\ddc v \wedge \theta_{\gamma_1}\wedge \cdots\wedge \theta_{\gamma_{n-1}}= \int_X v\,\ddc u \wedge \theta_{\gamma_1}\wedge \cdots\wedge \theta_{\gamma_{n-1}}.
\end{equation}
\end{theorem}
See Theorem~\ref{thm:intbypart1}. By a simple polarization, one can also get a slightly more general result, see Corollary~\ref{cor:ibpfinal}.

Our approach is to apply a construction due to Witt Nystr\"om. In \cite{WN17}, Witt Nystr\"om provided a construction that associates each potential $\varphi\in \PSH(X,\theta)$ with potentials $\Phi_N[\varphi]$ on $X\times \mathbb{P}^N$ with small unbounded loci for each $N\geq 1$. From the explicit construction, one can show that the Monge--Amp\`ere measures of $\Phi_N[\varphi]$ converge to the Monge--Amp\`ere measure of $\varphi$ in a proper sense. We study this construction and show that one can reduce the integration by parts formula to the special case using this construction.

As a corollary of our argument, we also get a formula that directly relates the integrals in \eqref{eq:1} to similar integrals on $X\times \mathbb{P}^N$. See Corollary~\ref{cor:intrestri}.

\subsection*{Acknowledgement}
We would like to thank Tamás Darvas and Antonio Trusiani for bringing this problem into our attention and for related discussions. We also thank Antonio Trusiani for pointing out the lack of one assumption in the statement of Theorem~\ref{thm:weak1} in the first arXiv version of this paper.
\section{Integration by parts}
\subsection{Notations}
Let $X$ be a compact K\"ahler manifold of dimension $n$. Let $\alpha$ be a big class with smooth representative $\theta$. Let $Z$ be the complement of the ample locus of $\alpha$ (\cite{Bou04}).
For each $N\geq 1$, define
\[
\Sigma_N:=\{\alpha\in \mathbb{R}^N_{\geq 0}: |\alpha|\leq 1\},
\]
where $|\alpha|$ is the sum of components of $\alpha$.

For each $N\geq 1$, we fix a basis $Z_0,\ldots,Z_N$ of $H^0(\mathbb{P}^N,\mathcal{O}(1))$. Let 
\[
H=H_N:=\{Z_0=0\}\subseteq \mathbb{P}^N.
\]

On $\mathbb{P}^N-H$, define
\[
z_a:=\frac{Z_a}{Z_0}\in \Gamma(\mathbb{P}^N-H,\mathcal{O}),\quad a=1,\ldots,N.
\]
We will identify $\mathbb{P}^N-H$ with $\mathbb{C}^N$ via $(z_1,\ldots,z_N)$.

Let $\omega_N$ be the Fubini--Study form on $\mathbb{P}^N$, normalized so that
\[
\int_{\mathbb{P}^N}\omega_N^N=1.
\]
By abuse of notation, we denote the metric induced by $\omega_N$ on $\mathcal{O}(1)$ by $\omega_N$. Take $C_0>0$ so that on $\mathbb{P}^N-H$,
\begin{equation}\label{eq:C0}
\omega_N=C_0 \ddc \log|Z_0|^2_{\omega_N}.
\end{equation}

For each $N\geq 1$, let 
\[
X_N:=X\times \mathbb{P}^N.
\]
Let $\pi_1^N,\pi_2^N$ be the natural projections:
\[
\begin{tikzcd}
  X_N \arrow[d,"\pi_1^N"] \arrow[r, "\pi_2^N"] & \mathbb{P}^N \\
  X &
\end{tikzcd}
\]
For simplicity, we denote $\pi_2^{N*}Z_A$ by $Z_A$ ($A=0,\ldots,N$), similar convention is used for $z_1,\ldots,z_N$. Similarly, we omit $\pi_1^{N*}$ from our notations from time to time.

Let
\[
\theta_N=\left(\pi^N_1\right)^*\theta+\left(\pi^N_2\right)^*\omega_N.
\]
Note that $[\theta_N]$ is a big class on $X_N$.

\subsection{Potentials with small unbounded loci}
For potentials $\varphi_1,\ldots,\varphi_p\in \PSH(X,\theta)$ ($p\leq n$), the non-pluripolar product 
\[
\theta_{\varphi_1}\wedge \cdots\wedge \theta_{\varphi_p}=\langle \theta_{\varphi_1}\wedge \cdots\wedge \theta_{\varphi_p} \rangle
\]
is a well-defined closed positive $(p,p)$-current on $X$ (\cite{BEGZ10} Proposition~1.6, Theorem~1.8). For a detailed study of these products, see \cite{BEGZ10}.

\begin{definition}
We say $\varphi\in \PSH(X,\theta)$ has \emph{small unbounded locus} if there is a closed  pluripolar subset $A\subseteq X$, such that $\varphi\in L^{\infty}_{\loc}(X-A)$.
\end{definition}
Examples of potentials with small unbounded loci include potentials with minimal singularities. There exists model potentials having non-vanishing Lelong number on a dense subset of $X$, and \emph{a fortiori} not having small unbounded loci (\cite{DDNL18c} Section~4).

\begin{proposition}\label{prop:smallun}
Let $\varphi_1,\ldots,\varphi_p\in \PSH(X,\theta)$ ($p\leq n$) be potentials with small unbounded loci. Let $A\subseteq X$ be a closed  pluripolar set, such that $\varphi_j\in L^{\infty}_{\loc}(X-A)$ for $j=1,\ldots,p$. Then
\[
\theta_{\varphi_1}\wedge \cdots \wedge \theta_{\varphi_p}=\mathds{1}_{X-A}\,\theta_{\varphi_1|_{X-A}}\wedge \cdots \wedge \theta_{\varphi_p|_{X-A}}.
\]
On RHS, the product is the usual Bedford--Taylor product (\cite{BT82}) on $X-A$, $\mathds{1}_{X-A}$ formally denotes the zero-extension to $X$.
\end{proposition}
See \cite{BEGZ10} Page~204.

Technically Proposition~\ref{prop:smallun} means that one can always reduce a local statement about potentials with small unbounded loci to a corresponding statement about locally bounded potentials, namely, to the Bedford--Taylor theory.

\begin{theorem}\label{thm:weak1}Let $p\leq n$. Let $\alpha_0,\ldots,\alpha_p$ be big classes on $X$ with smooth representatives $\theta_0,\ldots,\theta_p$.

Let $W\subseteq X$ be an open subset. Let $\chi\in C^{\infty}_0(W)$, $\chi\geq 0$.
Let $\Theta$ be a fixed closed positive $(n-p,n-p)$-current on $W$.

Let $\varphi_j^k,\varphi_j$ be $\theta_j$-psh functions on $W$ ($j=0,\ldots,p$ and $k\in \mathbb{Z}_{>0}$). Let $\psi^k,\psi$ be $\theta_0$-psh functions on $W$.
Assume that there is a closed pluripolar set $S$, such that $\varphi_j^k,\psi^k$ are uniformly bounded on each compact subset of $\Supp \chi-S$. Assume that $\varphi_j^k$ decreases to $\varphi_j$ for any $j$ and that $\psi^k$ decreases to $\psi$. Assume furthermore that 
$\varphi_0^k-\psi^k$ is uniformly bounded on $\Supp \chi-S$. Finally assume that
\[
\chi\,\theta_{1,\varphi_1^k}\wedge \cdots \wedge \theta_{p,\varphi_p^k}\wedge \Theta\to \chi\,\theta_{1,\varphi_1}\wedge \cdots \wedge \theta_{p,\varphi_p}\wedge \Theta
\]
weakly on $W$.
Then
\[
\chi(\varphi_0^k-\psi^k)\,\theta_{1,\varphi_1^k}\wedge \cdots \wedge \theta_{p,\varphi_p^k}\wedge \Theta\to \chi(\varphi_0-\psi)\,\theta_{1,\varphi_1}\wedge \cdots \wedge \theta_{p,\varphi_p}\wedge \Theta
\]
weakly as signed measures on $W$ as $k\to\infty$.
\end{theorem}
\begin{proof}
The proof is almost identical to that of \cite{DDNL18c} Lemma~4.1. We need \cite{GZ17} Theorem~3.18 for the local convergence.
\end{proof}

\begin{theorem}\label{thm:weak2} Let $\alpha_1,\ldots,\alpha_n$ be big classes on $X$ with smooth representatives $\theta_1,\ldots,\theta_n$.
Let $\varphi_j^k,\varphi_j\in \PSH(X,\theta_j)$ ($j=1,\ldots,n$ and $k\in \mathbb{Z}_{>0}$). Assume that there is a closed pluripolar set $S$, such that $\varphi_j^k$ are uniformly bounded on each compact subset of $X-S$. Assume that $\varphi_j^k$ converges to $\varphi_j$ in capacity for any $j$. Let $f^k$ be a sequence of uniformly bounded quasi-continuous functions converging in capacity to a bounded quasi-continuous function $f$ on $X$.
Finally assume that
\[
\theta_{1,\varphi_1^k}\wedge \cdots \wedge \theta_{n,\varphi_n^k}\to \theta_{1,\varphi_1}\wedge \cdots \wedge \theta_{n,\varphi_n}
\]
weakly on $X$.
Then
\[
f^k\,\theta_{1,\varphi_1^k}\wedge \cdots \wedge \theta_{n,\varphi_n^k}\to f\,\theta_{1,\varphi_1}\wedge \cdots \wedge \theta_{n,\varphi_n}
\]
weakly as $k\to\infty$.
\end{theorem}
\begin{proof}
The proof is almost identical to that of \cite{DDNL18c} Lemma~4.1.
\end{proof}
\begin{remark}
One can also state Theorem~\ref{thm:weak2} in a local way as Theorem~\ref{thm:weak1}. We do not need the local version of Theorem~\ref{thm:weak2} in this paper.
\end{remark}

\begin{theorem}[\cite{BEGZ10} Theorem~1.14]
 Let $\varphi_1,\varphi_2,\psi_1,\psi_2\in \PSH(X,\theta)$. Assume that all four potentials have small unbounded loci. Let $u=\varphi_1-\varphi_2$, $v=\psi_1-\psi_2$. Let $\Theta$ be a closed positive $(n-1,n-1)$-current on $X$.
 Assume that $u,v\in L^{\infty}(X)$. Then
\[
\int_X u\,\ddc v \wedge \Theta= \int_X v\,\ddc u \wedge \Theta=-\int_X \mathrm{d}v\wedge \mathrm{d}^{\mathrm{c}}u\wedge \Theta.
\]
\end{theorem}
The exact meaning of each term is explained in \cite{BEGZ10} Page~214. In \cite{BEGZ10}, integrals of this kind are denoted by $\int_{X-A}$ instead of $\int_X$ with $A$ being the unbounded locus of these potentials.
\subsection{Witt Nystr\"om construction}
In \cite{WN17}, Witt Nystr\"om provides a construction that reduces a problem about a general potential in $\PSH(X,\theta)$ to potentials with small unbounded loci. We briefly review this construction here.

Fix $\eta\in \PSH(X,\theta)$ such that 
\begin{enumerate}
    \item $\eta\in C^{\infty}(X-Z)$.
    \item $\eta\leq 0$.
\end{enumerate}
Recall that $Z$ is the complement of the ample locus of $\alpha$.
We may even assume that $\eta$ has analytic singularity by \cite{Bou04} Theorem~3.17. 

Let $W\subseteq X$ be an open subset.  
Let $\varphi$ be a $\theta$-psh function on $W$. We define $\Phi_N[\varphi]\in \PSH(W\times \mathbb{P}^N,\theta_N)$ by
\[
\Phi_N[\varphi]:=\sups_{\!\!\!\alpha\in \Sigma_N} \left((1-|\alpha|)(\eta+C_0\log |Z_0|^2_{\omega_N})+|\alpha|\varphi+\sum_{a=1}^N \alpha_a C_0\log |Z_a|_{\omega_N}^2- \alpha^2\right),
\]
where $C_0$ is defined in \eqref{eq:C0}.

In particular, taking $\alpha=0$, we find that
\begin{equation}\label{eq:lo}
\Phi_N[\varphi]\geq \eta+C_0\log|Z_0|_{\omega_N}^2.
\end{equation}
So $\Phi_N[\varphi]$ has small unbounded locus if $W=X$.

Observe that $\Phi_N[\varphi]$ is increasing in $\varphi$.

Define $\hat{\alpha}_a=\hat{\alpha}_a[\varphi]:(W-Z)\times \mathbb{C}^N \rightarrow [-\infty,\infty)$ $(a=1,\ldots,N)$ by
\begin{equation}\label{eq:alphaa}
\hat{\alpha}_a:=\frac{C_0\log|z_a|^2+\varphi-\eta}{2}.
\end{equation}
Observe that $\hat{\alpha}_a$ is usc.

We define $\hat{\alpha}:(W-Z)\times \mathbb{C}^N \rightarrow [-\infty,\infty)^N$ by
\begin{equation}\label{eq:alpha}
\hat{\alpha}=(\hat{\alpha}_1,\ldots,\hat{\alpha}_N).
\end{equation}

\begin{lemma}\label{lma:phin}
Let $W\subseteq X$ be an open subset.  
Let $\varphi$ be a $\theta$-psh function on $W$.
\begin{equation}\label{eq:PhiN}
\Phi_N[\varphi]=\sup_{\alpha\in \Sigma_N} \left((1-|\alpha|)(\eta+C_0\log |Z_0|^2_{\omega_N})+|\alpha|\varphi+\sum_{a=1}^N \alpha_a C_0\log |Z_a|_{\omega_N}^2-\alpha^2\right)
\end{equation}
on $(W-Z)\times (\mathbb{P}^N-H)$. Moreover, on this set,
\begin{equation}\label{eq:PhiN1}
\Phi_N[\varphi]=C_0\log|Z_0|_{\omega_N}^2+\eta-g_N\circ \hat{\alpha}[\varphi],
\end{equation}
where $g_N$ is the function defined in Appendix~\ref{sec:quad}.
\end{lemma}
\begin{proof}
In order to prove \eqref{eq:PhiN}, it suffices to show that the RHS of \eqref{eq:PhiN} is usc on $(W-Z)\times (\mathbb{P}^N-H)$.

Since $\log|Z_0|_{\omega_N}^2$ is obviously continuous, it suffices to prove that the following function is usc on $(W-Z)\times \mathbb{C}^N$:
\[
\begin{split}
I:=& \sup_{\alpha\in \Sigma_N} \left((1-|\alpha|)\eta+|\alpha|\varphi+\sum_{a=1}^N \alpha_a C_0 \log |z_a|^2-\sum_{a=1}^N \alpha_a^2\right) \\
=&\eta-g_N\circ \hat{\alpha}
\end{split}
\]
by completing the square.

Since $\hat{\alpha}_a$ is usc and $g_N$ is continuous and decreasing (Proposition~\ref{prop:gdecrease}), we conclude that $I$ is usc. Moreover, \eqref{eq:PhiN1} is implied by our calculation.
\end{proof}

In particular, let $\varphi_j,\varphi$ $(j\in \mathbb{Z}_{>0})$ be $\theta$-psh functions on $W$.
If $\varphi_j$ converges to $\varphi$ outside a pluripolar set, then $\Phi_N[\varphi_j]$ also converges to $\Phi_N[\varphi]$ outside a pluripolar set.

\begin{theorem}[\cite{WN17}\footnote{In \cite{WN17}, the normalizing constant is missing.}]\label{thm:WN}
Let $\varphi$ be a $\theta$-psh function on $W\subseteq X$, $N\geq 1$, we have
\[
\pi^N_{1*}\theta_{N,\Phi_N[\varphi]}^{N+n}=\binom{N+n}{n}N\int_0^1 \theta_{(1-t)\eta+t\varphi}^nt^{N-1}\,\mathrm{d}t.
\]
\end{theorem}

Evaluating the beta function, we get
\[
N\int_0^1 \theta_{(1-t)\eta+t\varphi}^nt^{N-1}\,\mathrm{d}t=\sum_{j=0}^n \binom{n}{j}\frac{j! (n-j+N-1)! N}{(n+N)!} \,\theta_{\eta}^j\wedge \theta_{\varphi}^{n-j}.
\]
As $N\to\infty$, the only non-vanishing terms is that of $j=0$, so
\begin{equation}\label{eq:convmeasure}
\binom{N+n}{n}^{-1}\pi^N_{1*}\theta_{N,\Phi_N[\varphi]}^{N+n}\to \theta_{\varphi}^n
\end{equation}
in \emph{total variation}.

\subsection{Integration by parts}
In the sequel, for $\varphi_1,\varphi_2\in \PSH(X,\theta)$, we write $[\varphi_1]=[\varphi_2]$ if there is a constant $C>0$, such that
\[
\varphi_2-C\leq \varphi_1\leq \varphi_2+C.
\]
\begin{theorem}\label{thm:intbypart1}
Let $\gamma_1,\ldots,\gamma_{n-1}\in \PSH(X,\theta)$. Let $\varphi_j,\psi_j\in \PSH(X,\theta)$ ($j=1,2$). Let $u=\varphi_1-\varphi_2$, $v=\psi_1-\psi_2$. Assume that 
\[
[\varphi_1]=[\varphi_2],\quad [\psi_1]=[\psi_2].
\]
Then
\begin{equation}\label{eq:ibp1}
\int_X u\,\ddc v \wedge \theta_{\gamma_1}\wedge \cdots\wedge \cdots \wedge   \theta_{\gamma_{n-1}}= \int_X v\,\ddc u \wedge \theta_{\gamma_1}\wedge \cdots\wedge \cdots \wedge\theta_{\gamma_{n-1}}.
\end{equation}
\end{theorem}
\begin{remark}
Here 
\[
\ddc v \wedge \theta_{\gamma_1}\wedge \cdots\wedge \cdots\wedge \theta_{\gamma_{n-1}}:=\theta_{\psi_1}\wedge \theta_{\gamma_1}\wedge \cdots\wedge \cdots \wedge \theta_{\gamma_{n-1}}-\theta_{\psi_2}\wedge \theta_{\gamma_1}\wedge \cdots\wedge \cdots \wedge \theta_{\gamma_{n-1}}
\]
is understood as a signed measure. This definition is independent of the choice of $\theta_{\psi_1}$ and $\theta_{\psi_2}$ by \cite{BEGZ10} Proposition~1.4.
Other similar expressions are understood in the same way.
\end{remark}

Note that by polarization, we may assume that $\gamma_1=\cdots=\gamma_{n-1}=\gamma$. We want to show
\begin{equation}\label{eq:ibp2}
\int_X u\,\ddc v \wedge \theta_{\gamma}^{n-1}=\int_X v\,\ddc u \wedge \theta_{\gamma}^{n-1}.
\end{equation}

Let us fix several notations. We introduce two variables $a,b\in [0,1]$ with $b=1-a$. For an expression $f(a,b)$, we write
\[
[f(a,b)]_1=\partial_a|_{a=0}f(a,1-a)
\]
when the derivative in $a$ is the right upper derivative (i.e. Dini derivative).

Let $W\subseteq X-Z$ be an open subset.
Let $\psi_1,\psi_2,\gamma$ be $\theta$-psh functions on $W$.
For each $N\geq 1$, define
\[
A^N[a,b]:=\Phi_N[a\psi_1+b\gamma]-\Phi_N[a\psi_2+b\gamma].
\]
We do not mention $\psi_1,\psi_2,\gamma,W$ in the notation explicitly, but they will always be clear from the context.

\begin{proposition}\label{prop:anstudy} Let $W\subseteq X-Z$ be an open subset.
Let $\psi_1,\psi_2,\gamma$ be $\theta$-psh functions on $W$. Assume that 
\[
v:=\psi_1-\psi_2\in L^{\infty}_{\loc}(W),\quad \gamma\leq \psi_1,\quad \gamma\leq \psi_2.
\]
\begin{enumerate}
    \item 
    On $W\times \mathbb{C}^N$,
    \[
        A^N[a,b]=-g_N \circ \hat{\alpha}[a\psi_1+b\gamma]+g_N \circ \hat{\alpha}[a\psi_2+b\gamma].
    \]
    \item For $a\in [0,1)$, $1-a>\epsilon>0$, on $\hat{\alpha}^{-1}(\mathbb{R}^N)$, we have
    \[
    A^N[a+\epsilon,b-\epsilon]-A^N[a,b]=\frac{\epsilon}{2} (\gamma-\psi_1)L\circ \hat{\alpha}[a\psi_1+b\gamma]-\frac{\epsilon}{2} (\gamma-\psi_2)L\circ \hat{\alpha}[a\psi_2+b\gamma]+\mathcal{O}(\epsilon^2),
    \]
    where $L:\mathbb{R}^N\rightarrow \mathbb{R}$ is the piecewise linear bounded function defined in Appendix~\ref{sec:quad}. The $\mathcal{O}$-constant depends only on $N$.
    \item For $a\in [0,1]$, on $\hat{\alpha}^{-1}(\mathbb{R}^N)$, 
    \begin{equation}\label{eq:2}
    A^N[a,b]=-\frac{va}{2}L\circ \hat{\alpha}[\gamma]+\mathcal{O}(a^2),    
    \end{equation}
    where the $\mathcal{O}$-constant depends only on $N$. 
    \item 
    \[
    \left|A^N[a,b]\right|\leq a|v|.
    \]
\end{enumerate}
\end{proposition}
\begin{proof}
(1) This follows from Lemma~\ref{lma:phin}.

(2) Observe that
\[
\hat{\alpha}[a\psi_1+b\gamma]-\hat{\alpha}[(a+\epsilon)\psi_1+(b-\epsilon)\gamma]=\frac{\epsilon}{2}(\gamma-\psi_1)e,
\]
where $e=(1,\ldots,1)$. By assumption, $\gamma-\psi_1\leq 0$,
so (2) follows from Proposition~\ref{prop:second}.

(3) Note that \eqref{eq:2} is a special case of (2).

(4) This follows directly from definition.
\end{proof}

\begin{corollary}\label{cor:capconv}
Let $\psi_1,\psi_2,\gamma\in \PSH(X,\theta)$. Assume that 
\[
 [\psi_1]=[\psi_2],\quad \gamma\leq \psi_1,\quad \gamma\leq \psi_2.
\]
As $a\to 0+$, $A^N[a,b]$ converges to $0$ in capacity.
\end{corollary}
\begin{proof}
Let $v=\psi_1-\psi_2$.

We need to show that for each $\epsilon>0$,
\[
\lim_{a\to 0+}\Capa \left\{\left|A^N[a,b]\right|>\epsilon\right\}=0.
\]
By Proposition~\ref{prop:anstudy}, we can take $C=C(N)$ such that
\[
\left|A^N[a,b]+\frac{va}{2}L\circ \hat{\alpha}[\gamma]\right|\leq Ca^2.
\]
Take $a$ small enough, we can thus assume that $Ca^2<\epsilon/2$, then
\[
\left\{\left|A^N[a,b]>\epsilon\right|\right\}\subseteq \left\{\left|\frac{va}{2}L\circ \hat{\alpha}[\gamma]\right|>\frac{\epsilon}{2} \right\}.
\]
Take a constant $C_1$ so that $|L|\leq C_1$, then
\[
\left\{\left|A^N[a,b]>\epsilon\right|\right\}\subseteq\left\{ a|v|> \frac{\epsilon}{C_1}\right\}.
\]
But since $v$ is the difference of two $\theta$-psh functions,
\[
\lim_{a\to 0+}\Capa\left\{|v|>\frac{\epsilon}{C_1 a}\right\}=0.
\]
Here the capacity is still the capacity on $X_N$ instead of on $X$, we have omitted the pull-back notations.
\end{proof}

\begin{theorem}\label{thm:intbyparts3}
Let $\gamma_1,\ldots,\gamma_{n-1},\varphi_j,\psi_j\in \PSH(X,\theta)$ ($j=1,2$). Let $u=\varphi_1-\varphi_2$, $v=\psi_1-\psi_2$. Assume that 
\[
[\varphi_1]=[\varphi_2],\quad [\psi_1]=[\psi_2].
\]
Assume furthermore that $\varphi_1$ has small unbounded locus.

Then
\begin{equation}\label{eq:ibp3}
\int_X u\,\ddc v \wedge \theta_{\gamma_1}\wedge \cdots\wedge \cdots \wedge   \theta_{\gamma_{n-1}}= \int_X v\,\ddc u \wedge \theta_{\gamma_1}\wedge \cdots\wedge \cdots \wedge\theta_{\gamma_{n-1}}.
\end{equation}
\end{theorem}
We postpone the proof of this theorem and see how this theorem will imply Theorem~\ref{thm:intbypart1}.

We need several other lemmata.
\begin{lemma}\label{lma:intrestri}
Let $\gamma, \varphi_j,\psi_j\in \PSH(X,\theta)$ ($j=1,2$). Let $u=\varphi_1-\varphi_2$, $v=\psi_1-\psi_2$. Assume that
\[
[\varphi_1]=[\varphi_2],\quad [\psi_1]=[\psi_2]= [\gamma].
\]
Moreover, assume that 
\[
\gamma\leq \psi_2\leq \psi_1.
\]

Then
\begin{equation}\label{eq:intres}
\int_X u\,\ddc v\wedge \theta_{\gamma}^{n-1}=\lim_{N\to\infty} 
\frac{(n-1)!}{N^{n-1}}\left[\int_{X_N}A^N[a,b]\,\pi^{N*}_{1}\ddc u \wedge \theta_{N,\Phi_N[\gamma]}^{N+n-1} \right]_1.
\end{equation}
\end{lemma}
\begin{proof}
By Lemma~\ref{lma:localesti}, the limit on the RHS of \eqref{eq:intres} exists.

Note that 
\[
\Phi[a\psi_1+b\gamma]\geq \Phi[a\psi_2+b\gamma],
\]
so $A^N[a,b]\geq 0$. 

Define 
\[
I=\int_X u\,\ddc v\wedge \theta_{\gamma}^{n-1}.
\]
Then 
\[
I=\frac{1}{n}\left[\int_X u\left((a\theta_{\psi_1}+b\theta_{\gamma})^n-(a\theta_{\psi_2}+b\theta_{\gamma})^n \right)\right]_{1}.
\]
By Theorem~\ref{thm:WN},
\[
I=\frac{1}{n}\lim_{N\to\infty}\binom{N+n}{n}^{-1}\left[\int_{X_N}\pi^{N*}_{1}u \left( \theta_{N,\Phi_N[a\psi_1+b\gamma]}^{N+n}-\theta_{N,\Phi_N[a\psi_2+b\gamma]}^{N+n} \right)\right]_1.
\]
Here we have made use of the fact that the integral on RHS is polynomial in $a$ and $b$ of bounded degree to change the order of limit and $[\cdot]_1$.
Then
\begin{equation}\label{eq:intbypartstemp}
\begin{split}
nI
=&\lim_{N\to\infty}\binom{N+n}{n}^{-1}\left[\int_{X_N}\pi^{N*}_{1}u \,\ddc A^N[a,b]\wedge \sum_{r=0}^{N+n-1}\left(\theta_{N,\Phi_N[a\psi_1+b\gamma]}^{r}\wedge \theta_{N,\Phi_N[a\psi_2+b\gamma]}^{N+n-1-r}\right)\right]_1\\
=&\lim_{N\to\infty}\binom{N+n}{n}^{-1}\left[\int_{X_N} A^N[a,b]\,\ddc\pi^{N*}_{1}u\wedge\sum_{r=0}^{N+n-1}\left(\theta_{N,\Phi_N[a\psi_1+b\gamma]}^{r}\wedge \theta_{N,\Phi_N[a\psi_2+b\gamma]}^{N+n-1-r}\right)\right]_1,
\end{split}
\end{equation}
where on the second line, we perform the integration by parts. This is allowed by our assumption and by Theorem~\ref{thm:intbyparts3}.

For $r=0,\ldots,N+n-1$, define
\[
J_r[a,b]:=\int_{X_N} A^N[a,b]\,\pi^{N*}_{1}\theta_{\varphi_1}\wedge\theta_{N,\Phi_N[a\psi_1+b\gamma]}^{r}\wedge \theta_{N,\Phi_N[a\psi_2+b\gamma]}^{N+n-1-r}.
\]
Observe that $J_r$ is decreasing with respect to $r$. In fact,
\[
J_r[a,b]=\int_0^{\infty} \mathrm{d}t \int_{\{A^N[a,b]>t\}}\,\pi^{N*}_{1}\theta_{\varphi_1}\wedge\theta_{N,\Phi_N[a\psi_1+b\gamma]}^{r}\wedge \theta_{N,\Phi_N[a\psi_2+b\gamma]}^{N+n-1-r}.
\]
So it suffices to prove that the inner integral is decreasing with respect to $r$.
Then since $\Phi_N[a\psi_j+b\gamma]$ $(j=1,2)$ have the same singularity type, by \cite{DDNL18c} Proposition~3.1, we can apply the partial comparison theorem (\cite{DDNL18c} Proposition~3.5) to conclude.

We claim that
\[
J_r[a,b]-\int_{X_N} A^N[a,b]\,\pi^{N*}_{1}\theta_{\varphi_1} \wedge \theta_{N,\Phi_N[\gamma]}^{N+n-1}=o(a),\quad a\to 0+.
\]

By monotonicity in $r$, it suffices to prove this for $r=0$ and $r=n+N-1$. Since the two cases are parallel, we can assume $r=0$. In fact, by Lemma~\ref{lma:localesti1},
\begin{equation}\label{eq:a2}
\int_{X_N} A^N[a,b]\,\pi^{N*}_{1}\theta_{\varphi_1}\wedge \theta_{N,\Phi_N[a\psi_2+b\gamma]}^{N+n-1}-\int_{X_N} A^N[a,b]\,\pi^{N*}_{1}\theta_{\varphi_1} \wedge \theta_{N,\Phi_N[\gamma]}^{N+n-1}=\mathcal{O}(a^2).
\end{equation}
So our claim holds. Hence
\[
\left[J_r[a,b]\right]_1=\left[\int_{X_N} A^N[a,b]\,\pi^{N*}_{1}\theta_{\varphi_1} \wedge \theta_{N,\Phi_N[\gamma]}^{N+n-1}\right]_1.
\]
The same argument holds with $\varphi_1$ replaced by $\varphi_2$, so \eqref{eq:intbypartstemp} implies that
\[
nI=\lim_{N\to\infty}\binom{N+n}{n}^{-1}(N+n)\left[\int_{X_N} A^N[a,b]\,\ddc\pi^{N*}_{1} u \wedge \theta_{N,\Phi_N[\gamma]}^{N+n-1}\right]_1
\]
and \eqref{eq:intres} follows.
\end{proof}

\begin{lemma}\label{lma:localesti1}
Let $W\subseteq X-Z$ be an open set. 
Let $\gamma,\psi,\varphi,\psi_j$ ($j=1,2$) be $\theta$-psh functions on $X$. 
Assume that
\[
0\leq \psi-\gamma\in L^{\infty}_{\loc}(W),\quad v:=\psi_1-\psi_2\in L^{\infty}_{\loc}(W).
\]

Take $\chi\in C^{\infty}_0(W)$, $\chi\geq 0$.

Define
\[
I_{W,N}[a,b]:=\int_{W\times \mathbb{C}^N}\chi A^N[a,b]\,\pi^{N*}_{1}\theta_{\varphi} \wedge \theta_{N,\Phi_N[a\psi+b\gamma]}^{N+n-1}.
\]

Then
\begin{equation}\label{eq:iwn}
\begin{split}
I_{W,N}[a,b]=a\binom{N+n-1}{n-1}N\int_0^1 t^N \int_W \chi v\,\pi^{N*}_{1}\theta_{\varphi} \wedge \left((1-t)\theta_{\eta}+t\theta_{\gamma}\right)^{n-1} \,\mathrm{d}t\\
+\mathcal{O}(a^2).
\end{split}
\end{equation}
\end{lemma}
Note that $\gamma,\psi_1,\psi_2$ appear in the definition of $A^N[a,b]$. Also note that the coefficient of $a$ in \eqref{eq:iwn} is independent of the choice of $\psi$.

\begin{proof}
Since the problem is local, we may shrink $W$ when necessary. 
Let
\[
\gamma'=\gamma'[a,b]=a\psi+b\gamma.
\]
Then
\[
\gamma'\geq \gamma
\]
and
\begin{equation}\label{eq:inproofa}
\gamma'-\gamma=a(\psi-\gamma).
\end{equation}

\textbf{Step 1}. We claim that we may assume that $\psi_1,\psi_2,\gamma,\psi$ are smooth. 

To be more precise, take an open subset $W'\Subset W$ containing $\Supp \chi$.

Take sequences of smooth $\theta$-psh functions on $W$, say $\psi_j^k$ ($k\geq 1,j=1,2$) that decreases to $\psi_j$ as $k\to\infty$, we may assume that
\[
|\psi_1^k-\psi_2^k|
\]
are uniformly bounded  on $W'$ as well.

Let 
\[
A^N_k[a,b]:=\Phi_N[a\psi_1^k+b\gamma]-\Phi_N[a\psi_2^k+b\gamma]
\]
Note that $\Phi_N[a\psi_j^k+b\gamma]$ decreases to $\Phi_N[a\psi_j+b\gamma]$ outside a pluripolar set. By Proposition~\ref{prop:anstudy},
\[
\left|A^N_k[a,b]\right|\leq Ca.
\]
By dominated convergence theorem, we have
\[
I_{W,N}^k[a,b]:=\int_{W\times \mathbb{C}^N}\chi A^N_k[a,b]\,\pi^{N*}_{1}\theta_{\varphi} \wedge \theta_{N,\Phi_N[\gamma']}^{N+n-1}
\]
converges to $I_{W,N}[a,b]$. Similar reasoning applies to the coefficient of $a$ in \eqref{eq:iwn}. The $\mathcal{O}$-constant in \eqref{eq:iwn} can be taken to be independent of $k$ as we will see in Step 3, so we conclude that we may assume that both $\psi_1,\psi_2$ are smooth.

Similar reasoning applies to $\gamma$ and $\psi$. As in \cite{WN17} Page~7, we may assume that $\gamma,\psi$ are bounded on $W$.
The convergences along Demailly approximations now follow from Theorem~\ref{thm:weak1} and the argument in \cite{WN17} Page~7. (Note that we do not have to assume that $\varphi$ has small unbounded locus here!) 

Now $\Phi_N[\gamma']$ is $C^{1,1}$ on $W\times \mathbb{C}^N$
(See \cite{WN17} Page~5). 

\textbf{Step 2}.
We claim that the measure
\[
\ddc\Phi \wedge \theta_{N,\Phi_N[\gamma']}^{N+n-1}
\]
is supported on $V_N$ for any local $\theta_N$-psh function $\Phi$ on $W\times \mathbb{C}^N$.

Here
\[
V_N:=\hat{\alpha}[\gamma']^{-1}\mathring{\Sigma}_N\subseteq W\times \mathbb{C}^N.
\]
Note that $V_N$ depends on $a,b$.

Since the problem is local on $W\times \mathbb{C}^N$, we may take $\theta_N=0$ by adding to $\Phi_N[\gamma']$ and $\Phi$ a smooth function.
We may focus on an open subset $A\subseteq W\times \mathbb{C}^N$ on which $\Phi_N[\gamma']$ is bounded.

For $k\geq 0$ large enough, let $O_k=\{\Phi>-k\}$. Then by definition of the pluripolar product, it suffices to prove that
\[
\mathds{1}_{O_k} \ddc(\Phi|_{O_k}) \wedge (\ddc\Phi_N[\gamma']|_{O_k})^{N+n-1}
\]
supports on $V_N$. Replacing $\Phi$ by $\max\{\Phi,-k\}$, we may assume that $\Phi$ is bounded as well. By continuity of the Bedford--Taylor product, we may then assume that $\Phi$ is smooth.

In this case, it is well-known that
\[
(N+n)\ddc \Phi\wedge (\ddc\Phi_N[\gamma'])^{N+n-1}=\left(\Delta_{\ddc\Phi_N[\gamma']}\Phi\right) (\ddc\Phi_N[\gamma'])^{N+n}.
\]
As shown in \cite{WN17} Page~5, $(\ddc\Phi_N[\gamma'])^{N+n}$ is supported on $V_N$. This proves our claim.

\textbf{Step 3}.
By Step 2,
\[
I_{W,N}[a,b]=\int_{V_N \cap(W\times \mathbb{C}^N)}\chi A^N[a,b]\,\theta_{\varphi} \wedge \theta_{N,\Phi_N[\gamma']}^{N+n-1}.
\]
We have omitted $\pi_1^{N*}$ from our notation.

We calculate its value now. Note that
\[
\hat{\alpha}[\gamma']=\hat{\alpha}[\gamma]+\frac{a}{2}(\psi-\gamma)e,
\]
where $e=(1,\ldots,1)\in \mathbb{R}^N$. By Appendix~\ref{sec:quad}, the piecewise linear function $L$ has the same coefficients at $\hat{\alpha}[\gamma]$ and $\hat{\alpha}[\gamma']$.
So
\[
\left| L\circ \hat{\alpha}[\gamma']-L\circ \hat{\alpha}[\gamma]\right| \leq Ca |\psi-\gamma|,
\]
where $C$ depends only on $N$.

It follows that
\begin{equation}\label{eq:cl}
\int_{V_N\cap (W\times \mathbb{C}^N)} \chi v \left| L\circ \hat{\alpha}[\gamma']-L\circ \hat{\alpha}[\gamma]\right|\,\theta_{\varphi}\wedge \theta_{N,\Phi_N[\gamma']}^{N+n-1}\leq Ca
\end{equation}
for a constant $C$ independent of $a$.

So by Proposition~\ref{prop:anstudy},
\begin{equation}\label{eq:nec1}
\begin{split}
I_{W,N}[a,b]=&-\frac{a}{2}\int_{V_N \cap(W\times \mathbb{C}^N)}\chi v L\circ \hat{\alpha}[\gamma]\,\theta_{\varphi}\wedge \theta_{N,\Phi_N[\gamma']}^{N+n-1}+\mathcal{O}(a^2)\\
 =&-\frac{a}{2}\int_{V_N \cap(W\times \mathbb{C}^N)}\chi v L\circ \hat{\alpha}[\gamma']\,\theta_{\varphi} \wedge \theta_{N,\Phi_N[\gamma']}^{N+n-1}+\mathcal{O}(a^2).
\end{split}
\end{equation}
But it is easy to see that on $V_N$, 
\[
-\frac{1}{2}L\circ \hat{\alpha}[\gamma']=\left|\hat{\alpha}[\gamma']\right|.
\]
Hence
\[
I_{W,N}[a,b]=a\int_{V_N \cap(W\times \mathbb{C}^N)} \chi v \left|\hat{\alpha}[\gamma']\right|\,\theta_{\varphi} \wedge \theta_{N,\Phi_N[\gamma']}^{N+n-1}+\mathcal{O}(a^2).
\]

Note that on $V_N$,
\[
\Phi_{N}[\gamma']=C_0\log|Z_0|^2_{\omega_N}+\eta+\hat{\alpha}^2,
\]
where
\[
\hat{\alpha}:=\hat{\alpha}[\gamma'].
\]

As in \cite{WN17} Page~6, on $V_N\cap(W\times \mathbb{C}^N)$,
\[
\theta_{N,\Phi_N[\gamma']}=(1-|\hat{\alpha}|)\theta_{\eta}+|\hat{\alpha}|\theta_{\varphi}+\omega_N+\sum_{a=1}^N\mathrm{d}\hat{\alpha}_a\wedge \mathrm{d}^{\mathrm{c}}\hat{\alpha}_a,
\]
For $x\in W$, define 
\[
V_x:=V_N \cap (\{x\}\times \mathbb{C}^N).
\]
Then as in \cite{WN17} Page~6\footnote{Note that there should be an extra $\binom{N+n}{n}$ on \cite{WN17} Page~6, line~5.},
\[
\begin{split}
I_{W,N}[a,b]=a\binom{N+n-1}{n-1}\int_W \chi(x) v(x)\int_{V_x}\omega_{N,\Phi_N[\gamma']|_{V_x}}^{N}(z) |\hat{\alpha}|(x,z)\\
\left(\theta_{\varphi} \wedge \left((1-|\hat{\alpha}|)\theta_{\eta}+|\hat{\alpha}|\theta_{\gamma'}\right)^{n-1}\right)(x)+\mathcal{O}(a^2).
\end{split}
\]

We can push-forward the integral to $\{x\}\times \mathbb{R}^N$ by the log map and pushing forward further to $\{x\}\times \Sigma_N$ by the gradient of $\Phi_N[\gamma'](x,k)$ as a function of $k\in \mathbb{R}^N$ as in \cite{WN17}, we get
\[
\begin{split}
I_{W,N}[a,b]=&a\binom{N+n-1}{n-1}N!\int_W \chi v\int_{\Sigma_N} |\hat{\alpha}| \mathrm{d}\hat{\alpha}\, \theta_{\varphi} \wedge \left((1-|\hat{\alpha}|)\theta_{\eta}+|\hat{\alpha}|\theta_{\gamma'}\right)^{n-1}+\mathcal{O}(a^2)\\
=&a\binom{N+n-1}{n-1}N\int_0^1 t^N \int_W \chi v\,\theta_{\varphi} \wedge \left((1-t)\theta_{\eta}+t\theta_{\gamma'}\right)^{n-1} \,\mathrm{d}t+\mathcal{O}(a^2)\\
=&a\binom{N+n-1}{n-1}N\int_0^1 t^N \int_W \chi v\,\theta_{\varphi} \wedge \left((1-t)\theta_{\eta}+t\theta_{\gamma}\right)^{n-1} \,\mathrm{d}t+\mathcal{O}(a^2),
\end{split}
\]
where the last line follows from \eqref{eq:inproofa}.
\end{proof}

\begin{lemma}\label{lma:localesti}
Let $W\subseteq X-Z$ be an open set. 
Let $\gamma,\psi,\varphi_j, \psi_j$ ($j=1,2$) be $\theta$-psh functions on $W$. 
Let
\[
u:=\varphi_1-\varphi_2,\quad v:=\psi_1-\psi_2.
\]
Assume that
\[
v\in L^{\infty}_{\loc}(W).
\]
Take $\chi\in C^{\infty}_0(W)$, $\chi\geq 0$.
Define
\[
I_W:=\lim_{N\to\infty}\frac{(n-1)!}{N^{n-1}}\left[\int_{W\times \mathbb{C}^N}\chi A^N[a,b]\,\ddc u \wedge \theta_{N,\Phi_N[\gamma]}^{N+n-1}\right]_1.
\]
Then $[\cdot]_1$ here is equal to the usual right derivative in $a$ and the limit exists and
\[
I_W=\int_W \chi v \,\ddc u \wedge \theta_{\gamma}^{n-1}.
\]
\end{lemma}

\begin{proof}
By Lemma~\ref{lma:localesti1},
\[
\begin{split}
I_W=&\sum_{j=0}^{n-1}\binom{n-1}{j}j!\lim_{N\to\infty} N \cdot \frac{(N+n-j-1)!}{(N+n)!}  \int_W \chi v\, \ddc u\wedge  \theta_{\eta}^j \wedge \theta_{\gamma}^{n-1-j}\\
=& \int_W \chi v\,\ddc u \wedge \theta_{\gamma}^{n-1}.
\end{split}
\]
Also by Lemma~\ref{lma:localesti1}, $[\cdot]_1$ here is equal to the usual right derivative in $a$.
\end{proof}

\begin{proof}[Proof of Theorem~\ref{thm:intbypart1}]
\textbf{Step 1}. We prove the theorem under the additional assumption that $[\gamma]=[\psi_2]$.

We may assume that 
\[
\gamma\leq \psi_2\leq\psi_1.
\]

By Lemma~\ref{lma:intrestri} and Lemma~\ref{lma:localesti}, it suffices to prove the theorem under additional assumptions that $\varphi_1$ has small unbounded locus, then this is exactly Theorem~\ref{thm:intbyparts3}.

\textbf{Step 2}. Now let us consider the general case. For any $a,b\in [0,1]$, $a+b=1$, we have by Step 1:
\[
\int_X u\,\ddc\left((a\psi_1+b\gamma)-(a\psi_2+b\gamma)\right)\wedge \theta_{a\psi_2+b\gamma}^{n-1}=\int_X \left((a\psi_1+b\gamma)-(a\psi_2+b\gamma)\right)\,\ddc u\wedge \theta_{a\psi_2+b\gamma}^{n-1}.
\]
Hence for $a>0$,
\[
\int_X u\,\ddc v\wedge \theta_{a\psi_2+b\gamma}^{n-1}=\int_X v\,\ddc u\wedge \theta_{a\psi_2+b\gamma}^{n-1}.
\]
Since both sides are polynomials in $a$, equality for all $a>0$ implies immediately equality at $a=0$. That is, 
\[
\int_X u\,\ddc v\wedge \theta_{\gamma}^{n-1}=\int_X v\,\ddc u\wedge \theta_{\gamma}^{n-1}.
\]
\end{proof}

\begin{proof}[Proof of Theorem~\ref{thm:intbyparts3}]
By polarization, we may assume that $\gamma_1=\cdots=\gamma_{n-1}=\gamma$.

We can repeat the same argument as in Lemma~\ref{lma:intrestri} and Theorem~\ref{thm:intbypart1}, the only differences being that 

1. now the integration by parts in \eqref{eq:intbypartstemp} follows from \cite{BEGZ10} Theorem~1.14. So we get  \eqref{eq:intres} as before. 

2. \eqref{eq:a2} is replaced by
\[
\int_{X_N} A^N[a,b]\,\pi^{N*}_{1}\theta_{\varphi_1}\wedge \theta_{N,\Phi_N[a\psi_2+b\gamma]}^{N+n-1}-\int_{X_N} A^N[a,b]\,\pi^{N*}_{1}\theta_{\varphi_1} \wedge \theta_{N,\Phi_N[\gamma]}^{N+n-1}=o(a).
\]
By Proposition~\ref{prop:anstudy}, it suffices to prove
\begin{equation}\label{eq:a1}
\int_{X_N} vL\circ\hat{\alpha}[\gamma]\,\theta_{\varphi_1}\wedge \theta_{N,\Phi_N[a\psi_2+b\gamma]}^{N+n-1}-\int_{X_N} vL\circ\hat{\alpha}[\gamma]\,\theta_{\varphi_1} \wedge \theta_{N,\Phi_N[\gamma]}^{N+n-1}=o(1),\quad a\to 0+.
\end{equation}
Note that $vL\circ\hat{\alpha}[\gamma]$ is quasi-continuous outside a closed pluripolar set: $v$ and $\hat{\alpha}[\gamma]$ are quasi-continuous (outside a closed pluripolar set). Since $L$ is continuous, $L\circ \hat{\alpha}[\gamma]$ is quasi-continuous as well. 
Now \eqref{eq:a1} follows from Theorem~\ref{thm:weak2} and Corollary~\ref{cor:capconv}.
\end{proof}

\begin{corollary}\label{cor:intrestri}
Let $\gamma, \varphi_j,\psi_j\in \PSH(X,\theta)$ ($j=1,2$). Let $u=\varphi_1-\varphi_2$, $v=\psi_1-\psi_2$. Assume that
\[
[\varphi_1]=[\varphi_2],\quad [\psi_1]=[\psi_2]=[\gamma].
\]
Moreover, assume that 
\[
\gamma\leq \psi_2\leq \psi_1.
\]

Then
\begin{equation}\label{eq:intres1}
\int_X u\,\ddc v\wedge \theta_{\gamma}^{n-1}=\lim_{N\to\infty} 
\frac{(n-1)!}{N^{n-1}}\left[\int_{X_N}u\,\ddc A^N[a,b] \wedge \theta_{N,\Phi_N[\gamma]}^{N+n-1} \right]_1.
\end{equation}
\end{corollary}
\begin{proof}
This follows from Lemma~\ref{lma:intrestri}, Theorem~\ref{thm:intbyparts3}, Theorem~\ref{thm:intbypart1}.
\end{proof}

Finally, let us observe that by a polarization procedure, one gets the following slightly more general result.
\begin{corollary}\label{cor:ibpfinal} Let $\alpha_j$ $(j=0,\ldots,n)$ be big cohomology classes on $X$. Let $\theta_j$ $(j=0,\ldots,n)$ be smooth representatives in $\alpha_j$. Let $\gamma_j\in \PSH(X,\theta_j)$ ($j=2,\ldots,n$). Let $\varphi_1,\varphi_2\in \PSH(X,\theta_0)$, $\psi_1,\psi_2\in \PSH(X,\theta_1)$. Let $u=\varphi_1-\varphi_2$, $v=\psi_1-\psi_2$. 
Assume that
\[
[\varphi_1]=[\varphi_2],\quad [\psi_1]=[\psi_2].
\]
Then
\begin{equation}\label{eq:ibp4}
\int_X u\,\ddc v \wedge \theta_{2,\gamma_2}\wedge \cdots\wedge \cdots \wedge   \theta_{n,\gamma_{n}}= \int_X v\,\ddc u \wedge \theta_{2,\gamma_2}\wedge \cdots\wedge \cdots \wedge\theta_{n,\gamma_{n}}.
\end{equation}
\end{corollary}
\begin{proof}
Take $\alpha=(a_0,\ldots,a_n)\in \mathbb{R}_{\geq 0}^{n+1}$. Let 
\[
\theta_{\alpha}=\sum_{j=0}^n a_j\theta_j.
\]
Let 
\[
\gamma_{\alpha}=\sum_{j=2}^n a_j\gamma_j,\quad \tilde{\gamma}_{\alpha}=a_0\varphi_1+a_1\psi_1+\gamma_{\alpha}.
\]
Note that
\[
\begin{split}
a_0u &= \left(a_0\varphi_1+a_1\psi_1+\gamma_{\alpha}\right)-\left(a_0\varphi_2+a_1\psi_1+\gamma_{\alpha}\right),\\
a_1v &= \left(a_0\varphi_1+a_1\psi_1+\gamma_{\alpha}\right)-\left(a_0\varphi_1+a_1\psi_2+\gamma_{\alpha}\right).
\end{split}
\]
So by Theorem~\ref{thm:intbypart1},
\[
a_0a_1 \int_X u\,\ddc v\wedge \theta_{\alpha,\tilde{\gamma}_{\alpha}}^{n-1}=a_0a_1 \int_X v\,\ddc u\wedge \theta_{\alpha,\tilde{\gamma}_{\alpha}}^{n-1}.
\]
So for $a_0>0,a_1>0$,
\[
\int_X u\,\ddc v\wedge \theta_{\alpha,\tilde{\gamma}_{\alpha}}^{n-1}= \int_X v\,\ddc u\wedge \theta_{\alpha,\tilde{\gamma}_{\alpha}}^{n-1}.
\]
Since both sides are polynomials in $a_0,\ldots,a_n$, this means that all coefficients are equal. In particular, the coefficients of
\[
a_2\ldots a_{n}
\]
are equal, hence proving \eqref{eq:ibp4}.
\end{proof}

\appendix
\section{Quadratic optimization}\label{sec:quad}
Let $N\geq 1$. We study the following function $f=f_N:\mathbb{R}^N\rightarrow \mathbb{R}$:
\[
f(x):=\min_{\alpha\in \Sigma_N} (x-\alpha)^2.
\]
Let $\Pi:\mathbb{R}^N\rightarrow \Sigma_N$ be the closest point projection. It is well-defined since $\Sigma_N$ is convex and closed.
Let $e=(1,1,\ldots,1)\in \mathbb{R}^N$.

Let $\mathcal{F}$ be the set of faces of $\Sigma_N$ as a simplex. By a face, we mean the interior of the face. The extremal points of $\Sigma_N$ are also considered as faces in $\mathcal{F}$.
So
\[
\Sigma_N=\coprod_{F\in \mathcal{F}}F.
\]
Observe that if $\Pi(x)\in F\in \mathcal{F}$, then so is $\Pi(x+\epsilon e)$ for small enough $\epsilon>0$. Let $A_F=\Pi^{-1}F$, then
\[
\mathbb{R}^N=\coprod_{F\in \mathcal{F}}A_F.
\]
Now observe that on each $A_F$, $\Pi$ is affine, say
\[
\Pi(x)=M^F x+c^F,
\]
where $M^F\in \mathfrak{gl}(N,\mathbb{R})$, $c^F\in \mathbb{R}^N$.

Define $g=g_N:\mathbb{R}^N\rightarrow \mathbb{R}$:
\[
g(x)=f(x)-x^2.
\]
Then we have
\[
g(x)= \left(x-\Pi x \right)^2-x^2.
\]

\begin{proposition}\label{prop:second}
For $x\in \mathbb{R}^N$,
\[
g(x+te)-g(x)=t L(x)+\mathcal{O}(t^2),\quad t\to 0+,
\]
where the $\mathcal{O}$-constant depends only on $N$, $L(x)$ is a bounded continuous piecewise linear function whose coefficients depend only on $N$.
\end{proposition}
\begin{proof}
All statements are obvious except that $L(x)$ is bounded and continuous. To see that $L$ is bounded, it suffices to show that $g(x+te)-g(x)$ is bounded for a fixed $t>0$. More generally, let $x,y\in \mathbb{R}^N$, then
\[
g(x)-g(y)\geq \min_{\alpha\in \Sigma_N} \left((x-\alpha)^2-x^2-(y-\alpha)^2+y^2\right)=\min_{\alpha\in \Sigma_N} 2\alpha\cdot (y-x).
\]
A similar inequality hold if we interchange $x$ and $y$. So
\[
|g(x)-g(y)|\leq C.
\]
To see $L$ is continuous, observe that
\[
g(x+te)-g(x)
\]
is a quadratic function in $t$ for any $x$. And since $L(x)$ is nothing but the coefficient of $t$, it suffices to show that $g(x+te)-g(x)$ is continuous in $x$ for three value of $t$. So the result follows from the obvious continuity of $g$.
\end{proof}
Now we extend the domain of definition of $g_N$, we will get a symmetric function $g_N:[-\infty,\infty)^N\rightarrow \mathbb{R}$. The definition is by induction on $N$, when $N=1$, we simply define
\[
g_1(-\infty)=0.
\]
For $N>1$, define
\[
g_N(x_1,\ldots,x_M,x_{M+1},\ldots,x_{N})=g_M(x_1,\ldots,x_M),
\]
where $x_{M+1},\ldots,x_N=-\infty$ and $x_1,\ldots,x_{M}\in \mathbb{R}$. We formally set $g_0=0$. We get a full definition of $g_N$ by requiring that it is symmetric in the $N$-arguments.
It is not hard to see that $g_N$ is continuous.
\begin{proposition}\label{prop:gdecrease}
The function $g_N:[-\infty,\infty)^N\rightarrow \mathbb{R}$ is decreasing in each of its arguments.
\end{proposition}
\begin{proof}
It suffices to prove this on $\mathbb{R}^N$. By definition, it suffices to show that for each $\alpha\in \Sigma_N$, the function
\[
(x-\alpha)^2-x^2
\]
is decreasing in each argument. This reduces immediately to the case $N=1$ and the result is obvious.
\end{proof}

\printbibliography

\bigskip
  \footnotesize

  Mingchen Xia, \textsc{Department of Mathematics, Chalmers Tekniska Högskola, G\"oteborg}\par\nopagebreak
  \textit{E-mail address}, \texttt{xiam@chalmers.se}

\end{document}